\documentclass[letterpaper, 10 pt, conference]{ieeeconf}  

\IEEEoverridecommandlockouts                              
\usepackage{graphics} 
\usepackage{epsfig} 
\usepackage{amsfonts,amsmath}
\usepackage{amssymb}

\usepackage{graphicx}
\usepackage{psfrag,epsfig,epstopdf}
\usepackage{hyperref}

\newtheorem{lem}{Lemma}
\newtheorem{thm}{Theorem}

\newtheorem{Def}{Definition}

\newtheorem{prop}{Proposition}

\newcommand{\tr}[1]{\operatorname{Tr}\left(#1\right)}

\newcommand{\EE}[1]{\text{\Large ${\mathbb E}$}\left[ #1 \right]}

\newcommand{\bra}[1]{\langle #1 |}
\newcommand{\ket}[1]{| #1 \rangle}

\newcommand{\NN}{{\mathbb N}}
\newcommand{\RR}{{\mathbb R}}

\newcommand{\ba}{\boldsymbol{a}}

\newcommand{\bI}{\boldsymbol{I}}

\newcommand{\bM}{\boldsymbol{M}}
\newcommand{\bN}{{\boldsymbol{N}}}

\newcommand{\bNp}{{\bN+1}}

\title{\LARGE \bf
Stabilization of photon-number states via single-photon corrections: \\ a first convergence analysis under an ideal set-up
\thanks{H.\ B.\ Silveira was fully supported by CNPq (National Counsel for the Scientific and Technological Development), Ministry of Science and Technology, Brazil, as a visiting professor at Centre Automatique et Syst\`emes, Mines ParisTech. P.\ S.\ Pereira da Silva was partially supported by CNPq. P.\ Rouchon was partially supported by Projet Blanc ANR-2011-BS01-017-01 EMAQS.}
}

\author{
H.\ B.\ Silveira\thanks{H.\ B.\ Silveira is with Departamento de Automa\c{c}\~ao e Sistemas (DAS), Federal University of Santa Catarina
(UFSC), Florian\'opolis, Brazil {\tt\small
hector.silveira@ufsc.br}}
\and  
P.\ S.\ Pereira da Silva \thanks{P.\ S.\ Pereira da Silva is with Escola Polit\'ecnica -- PTC, University of S\~ao Paulo (USP), S\~ao Paulo, Brazil {\tt\small paulo@lac.usp.br}}
\and
 P.\ Rouchon\thanks{P.\ Rouchon is with Centre Automatique et Syst\`emes, Mines ParisTech, PSL Research University,  Paris, France {\tt\small pierre.rouchon@mines-paristech.fr}}
}

\begin{document}

\bibliographystyle{plain}

\maketitle
\thispagestyle{empty}
\pagestyle{empty}

\begin{abstract}
This paper presents a first mathematical convergence analysis of a  Fock states feedback stabilization scheme via single-photon corrections. This measurement-based feedback has been developed and experimentally tested in 2012  by  the cavity  quantum  electrodynamics group of Serge Haroche and Jean-Michel Raimond.  Here, we consider the  infinite-dimensional Markov model corresponding to the  ideal set-up where detection errors and  feedback delays have been disregarded. In this ideal context, we show that any goal Fock state can be  stabilized by a Lyapunov-based feedback  for any  initial quantum state belonging to the  dense subset of finite rank density operators with  support  in a  finite photon-number sub-space.  Closed-loop simulations  illustrate the performance of the  feedback law.
\end{abstract}

\section{INTRODUCTION}

In~\cite{zhouPRL2012}, a photon-number  states (Fock state) feedback stabilization scheme via single-photon corrections was described and experimentally tested. Such control problem is relevant for quantum information applications \cite{nielsen-chang-book,haroche-raimondBook06}. The quantum state $\rho$ corresponds to the density operator of a  microwave field stored inside a super-conducting cavity and described as a quantum harmonic oscillator. At each sample step $k \in \NN$, a probe  atom is launched inside the cavity. The measurement outcome $y_k$ detected by a sensor is the energy-state of this  probe atom after its interaction with the microwave field. Each probe  atom  is  considered as a two-level system:  either it is detected in the  lowest energy state $\ket{g}$, or the highest energy state  $\ket{e}$. Consequently, the measurement outcomes  corresponds to a discrete-valued output $y_k$  with only  two distinct possibilities:  $g$ or $e$. Similarly,   the control  inputs  $u_k$ are also discrete-valued with 3 distinct  possibilities: $-1,0,+1$. The open-loop   value $u_k=0$  corresponds to a dispersive atom/field interaction: it  achieves in fact  a  Quantum Non-Demolition measurement of Fock states~\cite{brune-et-al:PhRevA92}. The two other   values $u_k=\pm 1$ correspond to resonant atom/field interactions where the probe atom and the field  exchange energy quanta: these values  achieve single-photon corrections.

Although the feedback law proposed and implemented in \cite{zhouPRL2012} considered imperfect detections on $y_k$ and delays in the control, here we focus on an ideal-set up, that is, detection errors and control delays have been disregarded. Theorem~\ref{thm:IdealConvergenceGeneral} shows that, by adding an arbitrarily small term to the Lyapunov function used in \cite{zhouPRL2012}, one ensures almost sure global stabilization of any goal Fock state for the closed-loop ideal set-up. This is achieved by relying on an infinite-dimensional Markov model of the ideal set-up that takes into account the back-action of the measurement outcome $y_k$ on the quantum state $\rho_{k+1}$.

Loosely speaking, in \cite{zhouPRL2012}, the control value $u_k$ at each sampling step $k$ was chosen so as to minimize the conditional expectation of the Lyapunov function $V(\rho_k)=\tr{d(\bN) \rho_k}$, where $\bN$ is the photon-number operator, $d(n) = (n-\overline{n})^2$ and $\overline{\rho}=\ket{\overline{n}}\bra{\overline{n}}$ is the goal Fock state. However, in closed-loop, the difference between such $V$ and its conditional expectation is not strictly positive: such  $V$ does not become a strict Lyapunov function in closed-loop  and additional  arguments  have to be considered to prove convergence. These additional arguments are related to Lasalle invariance. They are well established in a smooth context where the control $u$ is a smooth function of the state $\rho$. This cannot be the case here since $u$ is a discrete-valued control.    In order to overcome such technical difficulties,  we propose, similarly to \cite{AminiSDSMR2013A}, to add the arbitrarily small term $- \epsilon \sum_{n = 0}^{\infty} (\bra{n} \rho_k \ket{n})^2$ to $V(\rho_k)$, where $\epsilon > 0$. This slightly  modified control-Lyapunov function   becomes then  a strict-Lyapunov function in closed-loop that simplifies notably  the convergence analysis.   Moreover,  the  developed convergence analysis is done in  the infinite-dimensional setting in the following sense:  we show that,  for any  initial density operator $\rho_0$ with  a finite photon-number support  ($\rho_0 \ket{n} = 0$ for $n$ large enough), the closed-loop trajectory  $k\mapsto \rho_k$ remains also  with a finite photon-number support with a uniform bound on the maximum  photon-number. This  almost finite-dimensional behavior simplifies the convergence analysis despite the fact that such condition on $\rho_0$ is met on a dense subset of density operators (Hilbert-Schmidt topology on the Banach space of Hilbert-Schmidt self-adjoint operators).

The paper is organized as follows. Section~\ref{idealmodel} presents the ideal Markov model of the experimental set-up of the controlled microwave super-conducting cavity reported in~\cite{zhouPRL2012} and precisely formulates the Fock state stabilization problem here treated (see Definition~\ref{controlproblem}). Section~\ref{controlsolution} establishes the proposed solution to the control problem in two distinct parts. Firstly, Section~\ref{diagonalcase} considers the case where the initial condition $\rho_0$ is a diagonal density operator (see Theorem~\ref{thm:IdealConvergence}). Only the main ideas of the convergence proof are outlined. The technical details are given in Section~\ref{diagproof}. Afterwards, in Section~\ref{generalcase}, the main result of the paper is presented: the general solution is obtained from Theorem~\ref{thm:IdealConvergence} for $\rho_0$ belonging to a dense subset (see Theorem~\ref{thm:IdealConvergenceGeneral}). The simulation results are exhibited in Section~\ref{simus}. The proof of some intermediate results and computations required in Sections~\ref{controlsolution}~and~\ref{diagproof} are presented in Appendices~\ref{proof-properties-pk}--\ref{proofDelta}. Finally, the concluding remarks are given in Section~\ref{conclusions}.

\section{IDEAL  MARKOV MODEL}\label{idealmodel}
Denote by $\mathcal{H}$ the separable complex Hilbert space $L_2(\mathbb{C})$ with orthonormal basis $\{\ket n, n \in \mathbb{N}\}$ of Fock states (photon-number). Hence, $\mathcal{H} = \{ \sum_{n \in \NN} \psi_n \ket n,$ $(\psi_0, \psi_1, \dots) \in l_2(\mathbb{C})\}$. Let $\mathbb{D}$ be the set of all density operators on $\mathcal{H}$, that is, the set of  trace-class, self-adjoint, non-negative operators on $\mathcal{H}$ with unit trace. The sample step, corresponding to a sampling period around $100 \mu$s, is indexed by $k \in \NN=\{0, 1, 2, \dots \}$,  $u_k\in\{-1,0,1\}$ is the control,  $\rho_k \in \mathbb{D}$ the quantum  state and $y_k\in\{g,e\}$ the measurement outcome. The ideal Markov model of the controlled microwave super-conducting cavity used in~\cite{zhouPRL2012} is given by:
\begin{equation}\label{eq:IdealMarkovModel}
  \rho_{k+1}=
\left\{
  \begin{aligned}
  &\rho_{k+1}^g = \tfrac{\bM_g(u_k)\rho_k \bM_g^\dag(u_k)}{\tr{\bM_g(u_k)\rho_k \bM_g^\dag(u_k)}}  \mbox{ when } y_k=g,\\
  &\rho_{k+1}^e = \tfrac{\bM_e(u_k)\rho_k \bM_e^\dag(u_k) }{\tr{\bM_e(u_k)\rho_k \bM_e^\dag(u_k)}} \; \mbox{ when } y_k=e,
    \end{aligned}
  \right.
\end{equation}
where the measurements outcomes $y_k=g$ and $y_k=e$ occur with probabilities\footnote{As usual in quantum physics, it is here assumed that the measurement outcome $y_k=y$ cannot occur when $\tr{\bM_y(u_k)\rho_k \bM_y^\dag(u_k)}=0$, for $y=g,e$.}  $p_{g,k}=\tr{\bM_g(u_k)\rho_k \bM_g^\dag(u_k)}$ and $p_{e,k}=\tr{\bM_e(u_k)\rho_k \bM_e^\dag(u_k)}=1-p_{g,k}$, respectively,  $u_k=0$  corresponds to a dispersive interaction of the launched atom with the cavity field (Quantum Non-Demolition measurement of photons)
\begin{equation}\label{eq:M0}
  \bM_g(0) = \cos\left( \tfrac{\phi_0 \bN +\phi_R}{2}\right), \; \bM_e(0) \hspace{-2pt} = \hspace{-1pt} \sin\left(\hspace{-2pt} \tfrac{\phi_0 \bN +\phi_R}{2} \hspace{-2pt} \right) \hspace{-1pt},
\end{equation}
when $u_k=+1$ the atom enters the cavity in the state $\ket{e}$ with a resonant interaction with the cavity field
\begin{equation}\label{eq:Mp}
   \bM_g(+1)= \tfrac{\sin\left(\tfrac{\theta_0}{2}\sqrt{\bN}\right)}{\sqrt{\bN}} \ba^\dag, \;
    \bM_e(+1) \hspace{-2pt} = \hspace{-1pt} \cos\left(\hspace{-2pt} \tfrac{\theta_0}{2}\sqrt{\bNp} \hspace{-2pt} \right) \hspace{-1pt},
\end{equation}
when $u_k=-1$ it enters in $\ket{g}$ with a resonant interaction
\begin{equation}\label{eq:Mm}
    \bM_g(-1) = \cos\left(\tfrac{\theta_0}{2}\sqrt{\bN}\right), \;
       \bM_e(-1)= \ba \tfrac{\sin\left(\tfrac{\theta_0}{2}\sqrt{\bN}\right)}{\sqrt{\bN}},
\end{equation}
and $\phi_0, \phi_R, \theta_0 \in \RR$ are adjustable control parameters.
For each $u \in \{-1,0,1\}$, $\bM_g(u)$ and $\bM_e(u)$ are (linear) operators on $\mathcal{H}$ defined in the obvious way\footnote{For instance, $\bM_g(+1)\ket{n} = \left(\sin(\tfrac{\theta_0}{2}\sqrt{\bN})/\sqrt{\bN}\right)\sqrt{n+1}\ket{n+1} = \sin(\tfrac{\theta_0}{2} \sqrt{n+1}) \ket{n+1}$. In order for the definition of $\bM_e(-1)$ to be consistent, it is assumed $\sin(0)/0 = 1$.} according to the definitions in Appendix~\ref{basicoperators}. They are indeed well-defined operators on $\mathcal{H}$, despite the fact that $\ba$ and $\ba^\dag$ are unbounded operators. It is clear that $\bM_g(u), \bM_e(u)$ are bounded operators on $\mathcal{H}$ with $\bM_g^\dag(u)\bM_g(u) + \bM_e^\dag(u)\bM_e(u) = \bI$ (identity operator),
$\bM_e(-1)=\bM_g^\dag(+1)=\ba \sin (\tfrac{\theta_0}{2}\sqrt{\bN})/\sqrt{\bN}$, and $\bM_g(-1),\bM_g(0), \bM_e(0),\bM_e(+1)$ are self-adjoint.
It is easy to see that if the initial condition $\rho_0$ is a density operator then, for all realizations of the ideal Markov process~\eqref{eq:IdealMarkovModel}--\eqref{eq:Mm}, $\rho_k$ is a density operator for $k \in \mathbb{N}$.

Notice that $\overline{\rho}=\ket{\overline{n}}\bra{\overline{n}}$ is a steady state of the Markov process \eqref{eq:IdealMarkovModel}--\eqref{eq:Mm} with $u_k=0$, where $\overline{n} \in \NN$ is arbitrary.
The control problem here treated is given as follows:

\begin{Def}\label{controlproblem}
For the ideal Markov process~\eqref{eq:IdealMarkovModel}--\eqref{eq:Mm}, the control problem is to find a feedback law $u_k=f(\rho_k)$ such that, given an initial condition $\rho_0$ and $\overline{n} \in \NN$, the closed-loop trajectory $\rho_k$ converges almost surely towards the goal Fock state
$\overline{\rho}=\ket{\overline{n}}\bra{\overline{n}}$ as $k \rightarrow \infty$.
\end{Def}

The almost sure convergence above is with respect to the probabilities amplitudes $P_n(\rho) = \tr{\ket{n}\bra{n} \rho} = \bra{n} \rho \ket{n}$ of $\rho$, that is, $\lim_{k \to \infty} P_n(\rho_k) = P_n(\overline{\rho})$ for each $n \in \NN$.
In other words, $\lim_{k \to \infty} P_{\overline{n}}(\rho_k) = 1$ and $\lim_{k \to \infty} P_{n}(\rho_k) = 0$ when $n \neq \overline{n}$. The solution proposed in this paper for the control problem above is developed in the next section.

\section{STABILIZATION OF FOCK STATES}\label{controlsolution}

Given any operator $A$:~$\mathcal{H} \rightarrow \mathcal{H}$, let $A_{mn} = \bra{m} A \ket{n}$ for $m, n \in \NN$. Hence, $A_{nn}$ is the $n$-th diagonal element of $A$, while $A_{mn}$ with $m \neq n$ correspond to its ``off-diagonal" elements. One says that the operator $A$ is \emph{diagonal} when $A_{mn} = 0$ for all $m,n \in \NN$ with $m \neq n$.
One shall begin by solving the control problem given in Definition~\ref{controlproblem} in the particular case where the initial condition $\rho_0$ is diagonal (see Theorem~\ref{thm:IdealConvergence} in Section~\ref{diagonalcase}). Afterwards, in Section~\ref{generalcase}, the  solution to the general  non-commutative case  is presented (see Theorem~\ref{thm:IdealConvergenceGeneral}): its solution  relies essentially  on the diagonal case.

\subsection{Diagonal case}\label{diagonalcase}

For each $n^* \in \mathbb{N}$, define\footnote{Note that if $\rho=\ket{n} \bra{n}$ for some $n \in \NN$, then $\rho \in D_{n}$.}
\[
  D_{n^*} = \left\{ \rho \in \mathbb{D}  ~|~  \rho \mbox{ is diagonal and } \rho\ket{n}=0, \forall n > n^* \right\}.
\]
Consider the set $D_* = \bigcup_{n^* \in \NN} D_{n^*} \subset \mathbb{D}$.
 Note that $D_{n^*} \subset D_{n^*+1}$, and that each element $\rho$ of $D_*$ is ``finite
 dimensional'' in the following sense: $\rho \in \mathbb{D}$ is in $D_{n^*}$ if and only if $\rho = \sum_{n=0}^{n^*} \rho_{nn} \ket{n} \bra{n}$, and $\rho \in D_{n^*}$ may be considered as an operator from $\mathcal{H}$ to the finite-dimensional space $\mathcal{H}_{n^*}=\mbox{span}\{\ket{0}, \dots, \ket{n^*}\}$, or as a density matrix on $\mathcal{H}_{n^*}$. One defines the
 functions $n_{min}$:~$D_* \rightarrow \NN$,
 $n_{max}$:~$D_* \rightarrow \NN$
 and $n_{length}$:~$D_* \rightarrow \NN$ respectively by:
 \begin {itemize}
 \item $n_{min}(\rho)$ is the smallest $n \in \mathbb{N}$ such that $\rho\ket{n} \neq 0$;
 \item $n_{max}(\rho)$ is the greatest $n \in \mathbb{N}$ such that $\rho\ket{n} \neq 0$;
 \item $n_{length} (\rho) = n_{max}(\rho) - n_{min}(\rho)$.
 \end{itemize}
It is clear that, given $\rho \in D_*$, one has $\rho \in D_{n^*}$ if and only if $n_{max} (\rho) \leq n^*$.
The next result exhibits the properties of the state $\rho_k$ of \eqref{eq:IdealMarkovModel}--\eqref{eq:Mm} with respect to these functions.

\begin{prop}\label{properties-pk}
For every realization of the ideal Markov process~\eqref{eq:IdealMarkovModel}--\eqref{eq:Mm}
with initial condition $\rho_0 \in D_*$, one has that $\rho_k \in D_*$ for all
$k \in \NN$ with:
 \begin{itemize}
  \item If $u_k = 0$ or $u_k=-1$, then $n_{max} (\rho_{k+1}) \leq n_{max} (\rho_{k})$ and $n_{length} (\rho_{k+1}) \leq n_{length} (\rho_{k})$; \smallskip
  \item If $u_k = +1$, then $n_{max} (\rho_{k+1}) \leq n_{max} (\rho_{k}) +1$ and $n_{length} (\rho_{k+1}) \leq n_{length} (\rho_{k})$.
  \end{itemize}
\end{prop}
\begin{proof}
See Appendix~\ref{proof-properties-pk}.
\end{proof}

Take a goal photon-number $\overline{n} \in \NN$.
As in~\cite{AminiSDSMR2013A}, consider the following Lyapunov function $V_\epsilon$:~$D_* \rightarrow \RR$ defined as
\begin{equation}\label{lyapfundef}
V_\epsilon(\rho)= \tr{ d(\bN) \rho} - \epsilon \sum_{n\in\NN} \rho_{nn}^2,	\quad \mbox{for } \rho \in D_*,
\end{equation}
where $\epsilon >0$ is a real number and $d(n)=(n-\overline{n})^2$ as defined in~\cite{zhouPRL2012}.
The feedback law $u$:~$D_* \rightarrow \{-1, 0 ,1\}$ is given by
\begin{equation}\label{eq:feedback}
  u =f(\rho) \triangleq \underset{\upsilon\in\{-1,0,1\}}{\text{Argmin }}  \EE{V_\epsilon(\rho_{k+1})~|~\rho_k=\rho, u_k=\upsilon}.
\end{equation}
Note that for each $\rho \in D_*$ and $n^* \geq n_{max}(\rho)$, $d(\bN) \rho$ in \eqref{lyapfundef} is a well-defined self-adjoint, non-negative, trace-class operator on $\mathcal{H}$, by considering $d(\bN)$ as an operator on $\mathcal{H}_{n^*}$ and $\rho$ as an operator from $\mathcal{H}$ to $\mathcal{H}_{n^*}$.  Indeed, $d(\bN) \rho = \sum_{n=0}^{n^*} \rho_{nn} (n-\overline{n})^2 \ket{n} \bra{n}$. Thus,  \eqref{lyapfundef} is well-defined. Moreover, since $\mathcal{H}_{n^*}$ is invariant under $\rho \in  D_*$ for $n^* \geq n_{max}(\rho)$, it is clear that $\tr{ d(\bN) \rho} = \mbox{Tr}_{\mathcal{H}_{n^*}}(d(\bN) \rho)$, where on the right-hand side one considers $\rho$ as an operator on the finite-dimensional space $\mathcal{H}_{n^*}$ and the trace is taken over $\mathcal{H}_{n^*}$.

We have the following convergence result when $\rho_0 \in D_*$:

\begin{thm}\label{thm:IdealConvergence}
Let $\overline{n} \in \NN$ and $\epsilon > 0$. In \eqref{eq:M0}--\eqref{eq:Mm}, assume that $\phi_0/ \pi$ and $(\theta_0/\pi)^2$ are irrational numbers, and take $\phi_R=\pi/2 - \overline{n} \phi_0$. Consider the closed-loop Markov process~\eqref{eq:IdealMarkovModel}--\eqref{eq:Mm} with $u_k=f(\rho_k)$, where the feedback law $f$  is as in~\eqref{eq:feedback}. Then, given any initial condition $\rho_0 \in D_*$, one has that $\rho_k$ converges almost surely towards $\overline{\rho}=\ket{\overline {n}}\bra{\overline{n}}$ as $k \rightarrow \infty$.
\end{thm}

Its proof is decomposed into two steps:

\noindent \textbf{First Step.} Choose $\overline{n} \in \NN$ and $\epsilon > 0$. Let $n_0 = n_{length}(\rho_0)$, $r_0=n_{min}(\rho_0)$. Then, there exists an integer $m_0 > n_0 + r_0 + \overline{n} +1$ (depending on $n_0, r_0$, $\overline{n}$ and $\epsilon$)  such that, for all  closed-loop realizations $\rho_k$,  one has $\rho_k \in D_{m_0}$ for $k \in \NN$.
	
\noindent \textbf{Second Step.} Choose irrational numbers $\phi_0/ \pi$ and $(\theta_0/\pi)^2$ in \eqref{eq:M0}--\eqref{eq:Mm}, and take $\phi_R=\pi/2 - \overline{n} \phi_0$. In $D_{m_0}$, $V_\epsilon$ is a strict super-martingale: for all density operators  $\rho$  in  $D_{m_0}$, one has
\[
\EE{V_\epsilon(\rho_{k+1})~|~\rho_k=\rho, u_k=f(\rho)} -  V_\epsilon(\rho) = - Q_{V_\epsilon} (\rho, f(\rho)),
\]
where $Q_{V_\epsilon} (\rho, f(\rho))\geq 0$, and  $Q_{V_\epsilon} (\rho, f(\rho)) = 0$  if and only
if $\rho= \overline{\rho}$. The almost sure convergence follows then from usual results on strict super-martingales for Markov processes with compact state spaces.

The complete proof of the two steps above is presented in Section~\ref{diagproof}. The general case where the initial condition $\rho_0$ is not necessarily diagonal is treated in the next subsection.

\subsection{General case}\label{generalcase}

Consider, for each $n^* \in \NN$,
\[
  \mathbb{D}_{n^*} = \left\{ \rho \in \mathbb{D}  ~|~   \rho\ket{n} = 0, \forall n > n^* \right\} \subset \mathbb{D}_{n^*+1},
\]
and let $\mathbb{D}_* = \bigcup_{n^* \in \NN} \mathbb{D}_{n^*} \supset D_*$.
It is clear that $\rho \in  \mathbb{D}$ is in $\mathbb{D}_{n^*}$ if and only if
$\rho = \sum_{m,n=0}^{n^*} \rho_{mn} \ket{m} \bra{n}$. Consequently, $\mathbb{D}_*$ is a dense subset of $\mathbb{D}$ when $\mathbb{D}$ is endowed with the subspace topology induced from the Hilbert-Schmidt norm. Indeed, let $\mathcal{J}_2$ be the complex Banach space of all Hilbert-Schmidt operators on $\mathcal{H}$ with the Hilbert-Schmidt norm $\| B \|_2 = ( \sum_{m,n \in \NN} | B_{mn} |^2 )^{1/2}$, for $B \in \mathcal{J}_2$ \cite{reed-simon-1, conway-book}. Since $\mathbb{D} \subset \mathcal{J}_2$ and $\rho \in \mathbb{D}_{n^*}$ has the form $\rho = \sum_{m,n=0}^{n^*} \rho_{mn} \ket{m} \bra{n}$, the density property of $\mathbb{D}_*$ in $\mathbb{D}$ is clear.

One has that $\rho \in \mathbb{D}_{n^*}$ may be considered as an operator from $\mathcal{H}$ to the finite-dimensional space
$\mathcal{H}_{n^*}$, or as a density matrix on $\mathcal{H}_{n^*}$. Hence, $d(\bN) \rho$ is a well-defined trace-class operator on $\mathcal{H}$, by considering $d(\bN)$ as an operator on $\mathcal{H}_{n^*}$ and $\rho \in \mathbb{D}_{n^*}$ as an operator from $\mathcal{H}$ to $\mathcal{H}_{n^*}$. Indeed, $d(\bN) \rho = \sum_{m,n=0}^{n^*} \rho_{mn} (m-\overline{n})^2 \ket{m} \bra{n}$, and it is trace-class because its range is finite-dimensional \cite{reed-simon-1, conway-book}. Consequently,
 the Lyapunov function $V_\epsilon$ in \eqref{lyapfundef}, the feedback in \eqref{eq:feedback} and $n_{max}$ can be extended to $\mathbb{D}_*$.

Define the map $\Delta$:~$\mathbb{D}_* \rightarrow D_* \subset \mathbb{D}_*$ as $\Delta \rho = \sum_{n=0}^{n_{max}(\rho)} \rho_{nn} \ket{n} \bra{n}$. Note that $\Delta$ extracts the diagonal of $\rho \in \mathbb{D}_*$. It is easy to see that $n_{max}(\Delta \rho) = n_{max}(\rho)$ and $(\Delta \rho)_{nn} = \rho_{nn}$, $\rho \in \mathbb{D}_*$. Moreover, $\Delta \rho = \rho$ when $\rho \in D_*$. Other properties of the map $\Delta$ are given in the next result:

\begin{prop}\label{propertiesDelta}
Let $\rho \in  \mathbb{D}_*$,  $u \in \{-1, 0, 1\}$, $y=g,e$. Take $\alpha = \small {\tr{\bM_y(u)\rho \bM_y^\dag(u)}}$. Then:
\begin{itemize}
	\item $\tr{A \rho} = \tr{A \Delta \rho}$, for every diagonal bounded operator $A$:~$\mathcal{H} \rightarrow \mathcal{H}$; \smallskip
	\item $V_\epsilon(\rho) = V_\epsilon(\Delta \rho)$, for $\epsilon > 0$;  \smallskip
	\item $\alpha^{-1} \bM_y(u) \rho \bM_y^\dag(u)$ belongs to $\mathbb{D}_*$ with $\Delta \big(\alpha^{-1}\bM_y(u) \rho \bM_y^\dag(u)\big)= \alpha^{-1} \bM_y(u) (\Delta \rho) \bM_y^\dag(u)$;  \smallskip
	\item $\big[ \bM_y(u) (\Delta \rho) \bM_y^\dag(u) \big]_{nn} = \big[ \bM_y(u) \rho \bM_y^\dag(u) \big]_{nn}$, for all $n \in \NN$. In particular,
	$\alpha = \tr{\bM_y(u) (\Delta \rho) \bM_y^\dag(u)}$.
\end{itemize}
\end{prop}
\begin{proof}
See Appendix~\ref{proofDelta}.
\end{proof}

Now, let $\epsilon > 0$ and $\overline{\rho} = \ket{\overline{n}} \bra{\overline{n}}$, where $\overline{n} \in \NN$. Assume that $\rho_0 \in  \mathbb{D}_*$. Let $\rho_k$, $k \in \NN$, be the corresponding closed-loop trajectory for a fixed realization of \eqref{eq:IdealMarkovModel}--\eqref{eq:Mm} with feedback $u_k=f(\rho_k)$, where $f$ is as in~\eqref{eq:feedback}. It is immediate from the proposition above that:
\begin{itemize}
	\item $\rho_k \in \mathbb{D}_*$, for $k \in \NN$; \smallskip
	\item $\Delta \rho_k \in D_*$, $k \in \NN$, is the corresponding closed-loop trajectory of \eqref{eq:IdealMarkovModel}--\eqref{eq:Mm} for the initial condition $\Delta \rho_0$, the same realization (and with the same transition probabilities $p_{e,k}$ and $p_{g,k}$), as well as the same feedback $u_k = f(\rho_k) = f(\Delta \rho_k)$;
\smallskip
	\item $\tr{\ket{n} \bra{n} \rho_k} = \tr{\ket{n} \bra{n}  \Delta \rho_k}$, for any $n \in \NN$.
\end{itemize}

From these arguments, Theorem~\ref{thm:IdealConvergence} and the fact that $\Delta \overline{\rho} = \overline{\rho}$, one immediately obtains the following \emph{generic} solution to the control problem, that is, when the initial condition $\rho_0$ belongs to the dense subset $\mathbb{D}_*$ of $\mathbb{D}$:

\begin{thm}\label{thm:IdealConvergenceGeneral}
Let $\overline{n} \in \NN$ and $\epsilon > 0$. In \eqref{eq:M0}--\eqref{eq:Mm}, assume that $\phi_0/ \pi$ and $(\theta_0/\pi)^2$ are irrational numbers, and take $\phi_R=\pi/2 - \overline{n} \phi_0$. Consider the closed-loop Markov process~\eqref{eq:IdealMarkovModel}--\eqref{eq:Mm} with $u_k=f(\rho_k)$, where the feedback law $f$  is as in~\eqref{eq:feedback}. Then, given any initial condition $\rho_0 \in \mathbb{D}_*$, one has that $\rho_k$ converges almost surely towards $\overline{\rho}=\ket{\overline{n}}\bra{\overline{n}}$ as $k \rightarrow \infty$.
\end{thm}

\section{SIMULATION RESULTS}\label{simus}

This section presents the closed-loop simulation results concerning the application of Theorem~\ref{thm:IdealConvergenceGeneral} above to the ideal Markov process~\eqref{eq:IdealMarkovModel}--\eqref{eq:Mm}. The quantum experimental results exhibited in \cite{zhouPRL2012} used the following control parameter values in \eqref{eq:M0}--\eqref{eq:Mm}: $\phi_0 / \pi = 0.252$ and $\theta_0/\pi \approx 2/\sqrt{\overline{n} + 1}$. However, according to the assumptions in Theorem~\ref{thm:IdealConvergenceGeneral}, $\phi_0/ \pi$ and $(\theta_0/\pi)^2$ should be irrational numbers. Hence, here one chooses  $\phi_0/3.14 = 0.252$ and $\theta_0/3.14 =   2/\sqrt{\overline{n} + 1}$. One takes
$\rho_0 = \sum_{n=0}^{15} \ket{n}\bra{n}/16 \in \mathbb{D}_*$  as the initial condition, $\overline{n}=10$ for the goal Fock state
$\overline{\rho}=\ket{\overline{n}}\bra{\overline{n}}$, and $\epsilon = 10^3$ as the gain for the feedback $u_k=f(\rho_k)$ in \eqref{lyapfundef}--\eqref{eq:feedback}. Figure~\ref{sim1} exhibits the simulation results for one closed-loop realization with such choices and a final sample step of 120. It shows: the dynamics of the populations of $\rho_k$ (top), the controls $u_k$ (middle) and the simulated outcomes $y_k$ (bottom). The populations of $\rho_k$ correspond to the following observables: $A_1 = \sum_{n=0}^{\overline{n}-1} \ket{n}\bra{n}$ ($n < \overline{n}$), $A_2 = \ket{\overline{n}}\bra{\overline{n}}$ ($n = \overline{n}$), $A_3 = \sum_{n > \overline{n}} \ket{n}\bra{n}$ ($n > \overline{n}$). Therefore, one sees from the dynamics of the populations that $\rho_k$ converges to $\overline{\rho}$ as $k \rightarrow \infty$, which is in accordance with Theorem~\ref{thm:IdealConvergenceGeneral}. Note that $\bra{\overline{n}} \rho_k \ket{\overline{n}} \approx 1$ and $u_k = 0$ for all $k > 45$.

Recall that Theorem~\ref{thm:IdealConvergenceGeneral} assumes that $\epsilon > 0$. In order to further analyze the performance of the Lyapunov-based feedback law here proposed, we now make a comparison with the one used experimentally in \cite{zhouPRL2012}, which corresponds to take $\epsilon = 0$ in \eqref{lyapfundef}, i.e.\ to disregard the term $- \epsilon \sum_{n \in \NN} \rho_{nn}^2$. Figure~\ref{sim2} presents the simulation results for one closed-loop realization of such case. The control parameters, $\rho_0$ and $\overline{n}=10$ are the same as above. Note that $\bra{\overline{n}} \rho_k \ket{\overline{n}} \approx 1$ and $u_k = 0$ for all $k > 78$. In order to make a comparison in terms of the speed of convergence, define the settling time $k_s$ to be the smallest $\widetilde{k} \in \NN$ such that  $\bra{\overline{n}} \rho_k \ket{\overline{n}} > 0.9$ for all $k \geq \widetilde{k}$. One has $k_s=45$ for the case $\epsilon = 10^3$ above, and $k_s=78$ for $\epsilon = 0$. Therefore, in the two realizations here simulated, the choice of $\epsilon = 10^3$ reduced the settling time $k_s$ by nearly $42\%$ with respect to $\epsilon = 0$. This behavior is typical on an average basis, thereby justifying the term $- \epsilon \sum_{n \in \NN} \rho_{nn}^2$ in \eqref{lyapfundef}. Table~\ref{averagetable} shows the average value
$\overline{k}_s$ and the standard deviation $\sigma$ of $k_s$ for $\epsilon \in \{0, 0.1, 1, 10, 10^2, 10^3, 10^4, 10^5\}$, where a total of 5000 realizations were simulated for each $\epsilon$. Notice that when $\epsilon$ is relatively large or relatively small in comparison to $\epsilon = 10^3$, the average settling time $\overline{k}_s$ deteriorated. Furthermore, although for $\epsilon = 10^5$ one has that $\overline{k}_s$ increased by nearly $22\%$ in comparison to $\epsilon=10^3$, the standard deviation $\sigma$ decreased by nearly $62\%$. Computer simulations have suggested that a choice of $\epsilon > 0$ which may perhaps significantly improve $\overline{k}_s$ generally depends on the initial condition $\rho_0$ and on the goal Fock state $\overline{\rho}=\ket{\overline{n}}\bra{\overline{n}}$, and it has to be determined heuristically.

\begin{figure}[!hbt]
	\centering{\includegraphics[scale=0.5]{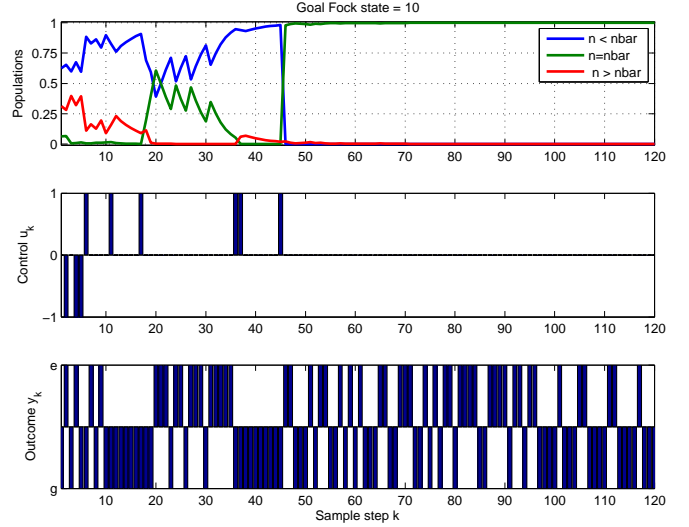}}
	\vspace{-0.7cm}
	\caption{Simulation of one closed-loop realization with gain $\epsilon = 10^3$: convergence of $\rho_k$ towards $\overline{\rho}$ (top), controls $u_k$ (middle), and outcomes $y_k$  (bottom). Notice that $\bra{\overline{n}} \rho_k \ket{\overline{n}} \approx 1$ and $u_k = 0$ for all $k > 45$.}\label{sim1}
\end{figure}

\begin{figure}[!hbt]
	\centering{\includegraphics[scale=0.5]{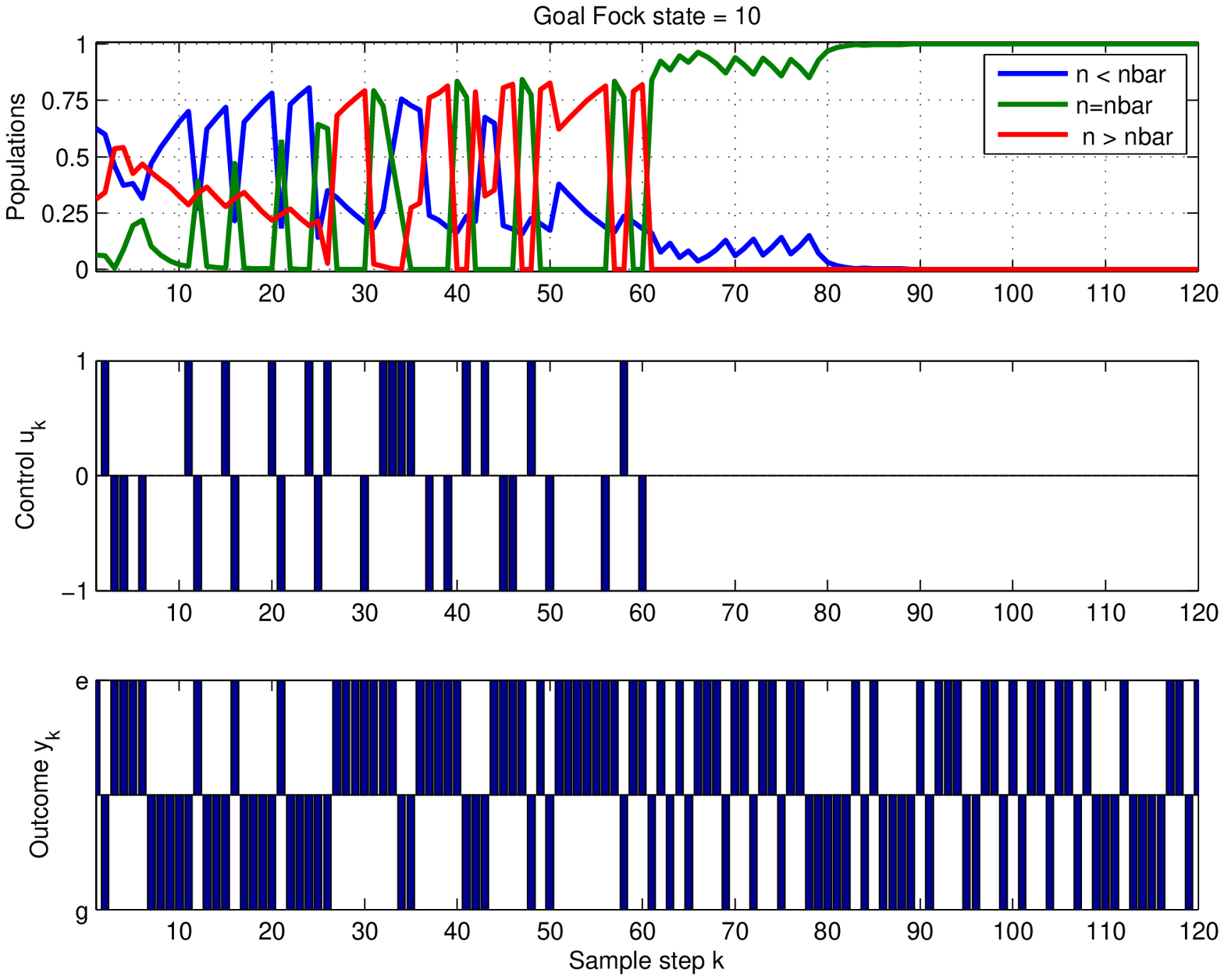}}
	\vspace{-0.7cm}
	\caption{Simulation of one closed-loop realization with gain $\epsilon = 0$: convergence of $\rho_k$ towards $\overline{\rho}$ (top), controls $u_k$ (middle), and outcomes $y_k$  (bottom). Notice that $\bra{\overline{n}} \rho_k \ket{\overline{n}} \approx 1$ and $u_k = 0$ for all $k > 78$.}\label{sim2}
\end{figure}

\begin{table}[!htb]
\caption{Average settling time $\overline{k}_{s}$ and standard deviation $\sigma$ as a function of the gains $\epsilon$, considering 5000 realizations}
\begin{center}
\begin{tabular}{| l | l | l | l | l |}\hline
	$\epsilon = 0$ & $\epsilon = 0.1$ & $\epsilon = 1$ &  $\epsilon =10$ \\
	$\overline{k}_{s}=79.94$ & $\overline{k}_{s}=79.95$  & $\overline{k}_{s}=81.24$ & $\overline{k}_{s}=71.33$ \\
	$\sigma=164.97$ & $\sigma=166.61$  & $\sigma=174.29$ & $\sigma=150.95$ \\ \hline \hline
	$\epsilon = 10^2$ & $\epsilon = 10^3$ & $\epsilon =10^4$ &  $\epsilon =10^{5}$  \\
	$\overline{k}_{s}=60.41$ & $\overline{k}_{s}=44.18$  & $\overline{k}_{s}=47.05$ & $\overline{k}_{s}=53.77$ \\
	$\sigma=119.39$ & $\sigma=44.12$  & $\sigma=37.37$ & $\sigma=16.84$ \\
\hline
\end{tabular}\label{averagetable}

\end{center}
\end{table}

\section{PROOF OF THEOREM~\ref{thm:IdealConvergence} (DIAGONAL CASE)}\label{diagproof}

\noindent \textbf{Proof of the First Step:}\\
Let $\epsilon > 0$. Define $V$:~$D_* \rightarrow \RR$ and $W$:~$D_* \rightarrow \RR$ as
\begin{equation}\label{defVW}
	V(\rho) =\tr{ d(\bN) \rho}, \qquad W(\rho) = -\sum_{n\in\NN} \rho_{nn}^2,
\end{equation}
respectively. Note that $V_\epsilon = V + \epsilon W$. Define:
\begin{itemize}
 	\item $Q_W (\rho, u) = W(\rho) - \EE{W(\rho_{k+1})~|~\rho_k = \rho, u_k=u}$,  \smallskip
  	\item $Q_V (\rho, u) = V(\rho) - \EE{V(\rho_{k+1})~|~\rho_k = \rho, u_k=u}$,  \smallskip
  	\item $Q_{V_\epsilon}(\rho,u) =  V_\epsilon(\rho)  - \EE{V_\epsilon(\rho_{k+1})~|~\rho_k = \rho, u_k=u}$,
\end{itemize}
for $\rho \in D_*$ and $u \in \{-1,0,1\}$. The proof of Theorem~\ref{thm:IdealConvergence} is a straightforward consequence of the next proposition:

 \begin{prop} \label{p1}
Let $\epsilon > 0$ and $n_0, r_0, \overline{n} \in \NN$.
There exists an integer $m_0 > n_0 + r_0 + \overline{n} + 1$ (depending on $\epsilon, n_0, r_0, \overline{n}$) such that, for each
$\rho \in D_*$ with $n_{length} (\rho) \leq n_0$, if $n_{max}(\rho) = m_0$, then
\[
	Q_{V_\epsilon} (\rho, -1) > \max \left\{ Q_{V_\epsilon} (\rho,0), Q_{V_\epsilon} (\rho,+1) \right\}.
\]
\end{prop}

In fact, given $\rho_0 \in D_*$, let $n_0 = n_{length}(\rho_0)$ and $r_0 = n_{min}(\rho_0)$. Note that $n_{max}(\rho_0) = n_0 + r_0 < m_0$.
By Proposition~\ref{properties-pk}, $\rho_k \in D_*$ with $n_{length}(\rho_k) \leq n_0$, for all $k \in \NN$. Since $u=f(\rho)$ maximizes $Q_{V_\epsilon} (\rho, f(\rho))$, Proposition~\ref{p1} implies that when $n_{max} (\rho_k) = m_0$ for some $k \in \NN$,  then the input $u_k$ will be always be
equal to $-1$, and hence Proposition~\ref{properties-pk} ensures that $n_{max} (\rho_{k+1}) \leq n_{max}(\rho_k) = m_0$. Therefore, $n_{max}(\rho_k) \leq m_0$, $k \in \NN$, showing the First Step.

The following two lemmas are instrumental for showing Proposition~\ref{p1}. Their proofs are given in Appendix~\ref{aLem1} and Appendix~\ref{aLem2}, respectively.

\begin{lem} \label{Lem1} Given an arbitrary nonzero $\theta_0 \in \RR$, fix any $a \in \RR$ such that $0 < a < 1/2$. For all  nonzero $N_0, N \in \NN$,
there exists an integer $\overline{N} > N$ big enough such that,
\[
	0 < 1/2 - a \leq \sin^2 \left( \tfrac{\theta_0}{2} \sqrt{n} \right) \leq 1/2 + a,
\]
for $n = \overline{N}, \overline{N}+1, \ldots,\overline{N} + N_0 - 1$.
\end{lem}

\begin{lem} \label{Lem2}
Let $\rho \in D_*$. Then:
  \begin{itemize}
\item $|Q_W (\rho, u)| \leq 1$, \quad  for each $u \in \{-1, 0, 1\}$; \smallskip
\item $Q_V (\rho, 0) = 0$; \smallskip
\item $Q_V (\rho, +1) = - \displaystyle \hspace{-2pt} \sum_{n \in \NN}  \hspace{-2pt} \rho_{nn} \hspace{-2pt} \left[ 2(n - \overline{n}) + 1 \right] \sin^2 \hspace{-2pt} \left( \tfrac{\theta_0}{2} \sqrt{n+1} \right)$; \smallskip
\item $Q_V (\rho, -1) = \displaystyle  \sum_{n \in \NN}  \rho_{nn} \left[ 2(n - \overline{n}) -1 \right] \sin^2 \left( \tfrac{\theta_0}{2} \sqrt{n} \right)$.
\end{itemize}
\end{lem}

The proof of Proposition~\ref{p1} is shown in the sequel. 
\begin{proof}
Let $\epsilon > 0$ and $n_0, r_0, \overline{n} \in \NN$. One has to show that there exists $m_0 > n_0 + r_0 + \overline{n} + 1$ such that, if $\rho \in D_*$ with $n_{length} (\rho) \leq n_0$, then $u=-1$
always maximizes $Q_{V_\epsilon}(\rho, u)$ whenever $n_{max}(\rho) = m_0$. From Lemma~\ref{Lem2} and the fact that $Q_{V_\epsilon} = Q_V + \epsilon Q_W$, to complete the proof it suffices to show that:
\begin{itemize}
	\item  If  $\rho \in D_*$ is such that $n_{length} (\rho) \leq n_0$ and $n_{max} (\rho) \geq n_0 + \overline{n}$, then $Q_V (\rho, +1) \leq 0$; \smallskip
	\item There exists $m_0 > n_0 + r_0  + \overline{n} +1$ such that $Q_V (\rho, -1) > 2 \epsilon$, whenever $\rho \in D_*$ is such that $n_{length} (\rho) \leq n_0$ and $n_{max}(\rho) = m_0$.
\end{itemize}
Note that
\[
	Q_V (\rho, +1) = - \hspace{-8pt} \sum_{n = n_{min} (\rho)}^{n_{max}(\rho)} \hspace{-2pt}  \rho_{nn}  [ 2(n - \overline{n}) +1 ] \sin^2 \left( \tfrac{\theta_0}{2} \sqrt{n+1} \right),
\]
for any $\rho \in D_*$. Thus, if $n_{length} (\rho) \leq n_0$ and $n_{max}(\rho) \geq \overline{n} + n_0$, then
$n_{min}(\rho) \geq \overline{n} $, and hence the first claim is shown.

Now, fix $0 < a < 1/2$ and let\footnote{As $N$ is an integer, it follows that $N \geq \overline{n} + 1$.}
$N \geq \frac{1}{2} \left[ \frac{2 \epsilon}{1/2-a} +  2\overline{n} + 1\right]$.
Applying Lemma \ref{Lem1} for $N_0 = n_0 + r_0 + 1$ and such choice of $N$, one gets $\overline{N}  > N$ in which 	
$0 < 1/2 - a \leq \sin^2 \left( \tfrac{\theta_0}{2} \sqrt{n} \right)$,
for $n = \overline{N}, \overline{N}+1, \ldots,\overline{N} + n_0 + r_0$.
Take $m_0 = \overline{N} + n_0 + r_0$. Let $\rho \in D_*$ with $n_{length} (\rho) \leq n_0$ and $n_{max}(\rho) = m_0$. Note that $m_0 > n_0 + r_0 + \overline{n} +1$ and $n_{min} (\rho) \geq \overline{N} + r_0$. From Lemma~\ref{Lem2} and the inequality above for $1/2 - a$, one obtains
\[
\begin{array}{l}
	Q_V (\rho, -1)  =  \displaystyle \sum_{n=n_{min}(\rho)}^{m_0}   \rho_{nn} [ 2(n - \overline{n}) -1 ]  \sin^2 \left( \tfrac{\theta_0}{2} \sqrt{n} \right)\\
   	 \geq   \displaystyle \sum_{n=n_{min}(\rho)}^{m_0} \rho_{nn}  [ 2(n - \overline{n}) -1  ] (1/2-a) \\
   	 \geq \displaystyle  \sum_{n=n_{min}(\rho)}^{m_0} \rho_{nn}  [ 2(\overline{N} - \overline{n}) -1 ] (1/2-a) \\
   	 =  \displaystyle [ 2(\overline{N} - \overline{n}) -1 ](1/2-a) \sum_{n=n_{min}(\rho)}^{m_0} \rho_{nn}.
\end{array}
\]
Using the fact that $\sum_{n=n_{min}(\rho)}^{m_0} \rho_{nn}  =1$ and $\overline{N} > \frac{1}{2} \left[
\frac{2 \epsilon}{1/2-a} + 2\overline{n} + 1\right]$, one shows the second claim, thereby completing the proof of Proposition~\ref{p1}.
\end{proof}

\vspace{0.2cm}

\noindent \textbf{Proof of the Second Step:}\\
Let $\epsilon > 0$. Recall that, by definition, $Q_{V_\epsilon} = Q_V + \epsilon Q_W$. Using the same notation of the First Step, the central idea of the proof is to show that, given $\rho \in D_{m_0}$, one has that $Q_{V_\epsilon}(\rho, f(\rho)) \geq 0$, and that $Q_{V_\epsilon}(\rho, f(\rho))=0$ if and only if $\rho = \overline{\rho}$.
The following lemma is instrumental for the proof of such property. Its proof is presented in Appendix~\ref{aLem3}.
\begin{lem}\label{Lem3}
Assume that $\phi_0/ \pi$ is an irrational number in \eqref{eq:M0}, and take $\phi_R=\pi/2 - \overline{n} \phi_0$, where $\overline{n} \in \NN$. Let $\rho \in D_*$. Then:
\begin{itemize}
	\item $Q_W(\rho,0) \geq 0$, and $Q_W(\rho, 0) =0$ if and only if $\rho = \ket{n} \bra{n}$ for some $n \in \NN$; \smallskip
  	\item $Q_W(\rho,+1) = Q_W(\rho,-1) = 0$ whenever $\rho = \ket{n} \bra{n}$ for some $n \in \NN$.
\end{itemize}
\end{lem}

One has that $m_0 > \overline{n}$, and $(\theta_0 / \pi)^2$ is an irrational number by assumption. Recall that $\sin^2(x) = 0$ if and only if $x=\ell \pi$, where $\ell$ is an integer.
First we show that $Q_{V_\epsilon}(\overline{\rho}, f(\overline{\rho})) = 0$. By Lemma~\ref{Lem2}: $Q_V(\overline{\rho}, +1) = -\sin^2 ( \tfrac{\theta_0}{2} \sqrt{\overline{n}+1}) < 0$; $Q_V(\overline{\rho}, -1) = -\sin^2(\tfrac{\theta_0}{2} \sqrt{\overline{n}}) < 0$ when $\overline{n} > 0$, and $Q_V(\overline{\rho}, -1)=0$ when $\overline{n} =0 $; and $Q_V(\overline{\rho}, 0) =0$. As $Q_W(\overline{\rho}, u) =0$, for
$u \in \{-1, 0 ,1 \}$, and $u=f(\overline{\rho})$ maximizes $Q_{V_\epsilon}(\overline{\rho}, u)$, one has that $Q_{V_\epsilon}(\overline{\rho}, f(\overline{\rho})) =0$.

Now, let $\rho \in D_{m_0} \subset D_*$. Since $u=f(\rho)$ maximizes $Q_{V_\epsilon}(\rho,u)$, it follows that
\[
   Q_{V_\epsilon}(\rho, f(\rho)) \geq Q_{V}(\rho, 0) + \epsilon Q_{W} (\rho, 0) = \epsilon Q_{W} (\rho, 0) \geq 0.
\]
Suppose $Q_{V_\epsilon}( \rho, f(\rho)) =0$. Hence, $Q_{W} (\rho, 0) = 0$, and so $\rho = \ket{n} \bra{n}$ for some $n \in \{0, 1, \ldots, m_0\}$.
It suffices to show that $Q_{V_\epsilon}(\rho, f(\rho))|_{\ket{n}\bra{n}} > 0$ for $n \in  \{0, 1, \ldots, m_0\}$ with $n \neq \overline{n}$.
Assume that $n > \overline{n}$.  It is clear that $u =f(\ket{n} \bra{n}) = -1$ and
$Q_{V_\epsilon}(\rho, f(\rho))|_{\ket{n}\bra{n}} =  [2(n-\overline{n}) - 1] \sin^2(\tfrac{\theta_0}{2} \sqrt{n})  > 0$. Assume now that $n < \overline{n}$.
Then, $u =f(\ket{n} \bra{n}) = +1$ and $Q_{V_\epsilon}(\rho, f(\rho))|_{\ket{n}\bra{n}} = -[2(n-\overline{n})+1] \sin^2(\tfrac{\theta_0}{2} \sqrt{n+1})>0$. This completes the proof of the referred property.

The remaining part of the proof of the Second Step is a straightforward consequence of the standard stochastic convergence result below:
\begin{thm}\cite[Theorem~1, p.\ 195]{kushner-71}\label{convtheorem}
Let $\Omega$ be a probability space and let $W$ be a measurable space. Consider that
$X_k$:~$\Omega \rightarrow W$, $k \in \NN$, is a Markov chain with respect
to the natural filtration. Let $Q$:~$W \rightarrow \RR$ and $V$:~$W \rightarrow \RR$ be
measurable non-negative functions with $V(X_k)$ integrable for all
$k \in \NN$. If
$\EE{V(X_{k+1}) ~|~ X_k } - V(X_k) = - Q(X_k)$, for  $k \in \NN$,
then $\lim_{k \to \infty} Q(X_k) = 0$ almost surely.
\end{thm}

Indeed, let $\mathcal{J}_1$ be the complex Banach space of all trace-class operators on $\mathcal{H}$ with the trace norm $\| \cdot \|_1$, that is, $\| B \|_1 = \tr{|B|}$, where $| B | \triangleq \sqrt{B^\dag B}$, for $B \in \mathcal{J}_1$. Recall that $\| B \| \leq \| B \|_1$ and\footnote{One also recalls that if $A$ is a bounded operator on $\mathcal{H}$ and $B \in \mathcal{J}_1$, then $AB, BA \in \mathcal{J}_1$ with $\tr{AB}=\tr{BA}$.} $| \tr{A B} | \leq \|A\| \| B \|_1$, for every $B \in \mathcal{J}_1$ and each bounded operator $A$:~$\mathcal{H} \rightarrow \mathcal{H}$, where $\| \cdot \|$ is the usual operator norm (\emph{sup} norm of bounded operators) \cite{reed-simon-1, conway-book}.
Consider the subspace topology on $D_{m_0}$ with respect to $\mathcal{J}_1$. One has that the closed-loop trajectory $\rho_k$, $k \in \NN$, is a Markov chain with phase space $D_{m_0}$ (with respect to the natural filtration and the Borel algebra on $D_{m_0}$). It is clear that $D_{m_0}$ is compact, and that $Q_{\epsilon}$ and $V_{\epsilon} - \alpha_\epsilon$ are non-negative and continuous on $D_{m_0}$, for all $\epsilon > 0$, where $\alpha_\epsilon \triangleq \min_{\rho \in D_{m_0}} V_\epsilon(\rho)$.
The theorem above implies that $\rho_k$ converges almost surely towards $\overline{\rho}$ as $k \rightarrow \infty$ (with respect to the trace norm). This completes the proof of Theorem~\ref{thm:IdealConvergence}.

\section{CONCLUDING REMARKS}\label{conclusions}

This paper provided a convergence analysis of Fock states stabilization via single-photon corrections under an ideal set-up, that is, assuming perfect measurement detection and no control delays. In terms of convergence speed, the simulation results here presented have justified the inclusion of the term $- \epsilon \sum_{n \in \NN} \rho_{nn}^2$ in the Lyapunov-based feedback law \eqref{lyapfundef}--\eqref{eq:feedback}. It is straightforward to verify that the convergence analysis developed in this paper remains valid  for: (i) any other function $d(n)$ in \eqref{lyapfundef} satisfying $d(\overline{n})=0$, $d(n)$ is increasing for $n > \overline{n}$ and $d(n)$ is decreasing for $n < \overline{n}$; and (ii) $\epsilon > 0$ dependent on $n$, that is, to take the term $- \sum_{n \in \NN} \epsilon_{n} \rho_{nn}^2$. However, it is an open problem how to choose the function $d(n)$ and the gains $\epsilon_n > 0$ so as to achieve the best convergence speed.

Finally, the feedback law used in \cite{zhouPRL2012}, which corresponds to $\epsilon = 0$, was tailored for an experimental set-up with measurement imperfections and control delays. The convergence analysis of such realistic situation will be investigated in the future.


\section{ACKNOWLEDGMENTS}
The authors are indebted I.\ Dotsenko,  M.\ Brune and J.\ M.\ Raimond for valuable discussions on the experimental feedback scheme.

\appendix

\subsection{Basic properties of the operators $\bN$, $\ba$ and $\ba^\dag$}\label{basicoperators}


Fix $n^* \in \NN$ and let  $\mathcal{H}_{n^*}=\mbox{span}\{\ket{0}, \dots, \ket{n^*}\}$. Consider the (linear) operators $\bN$:~$\mathcal{H}_{n^*} \rightarrow \mathcal{H}_{n^*}$, $\ba$:~$\mathcal{H}_{n^*} \rightarrow \mathcal{H}_{n^*-1} \subset \mathcal{H}_{n^*}$, $\ba^\dag$:~$\mathcal{H}_{n^*} \rightarrow \mathcal{H}_{n^*+1}$
defined respectively as $
\bN \ket{n}= n~\ket{n}$,
$\ba \ket{0}=0$,  $\ba \ket{n} =  \sqrt{n}~\ket{n-1}$ for $n\geq 1$, $ \ba^\dag \ket{n} =\sqrt{n+1}~\ket{n+1}$. Note that these operators cannot be extended to $\mathcal{H}$.
Let $f$:~$\NN \rightarrow \RR$ be a function. Define
the operator $f(\bN)$:~$\mathcal{H}_{n^*} \rightarrow \mathcal{H}_{n^*}$ by
$f(\bN) \ket{n} = f(n) \ket{n}$, for each $n=0,\dots,n^*$.
It is clear that $f(\bN)$ can be extented to $\mathcal{H}$ whenever $f$ is a bounded function.
Given $f$:~$\NN \rightarrow \RR$ and an integer $m$, one defines $g$:~$\NN \rightarrow \RR$ as:
$g(n) = f(n+m)$, when $n + m \geq 0$; and $g(n)=0$, when $n+m < 0$. One abuses notation letting
$f(\bN+m)$ stand for $g(\bN)$. Given two functions $f,g$:~$\NN \rightarrow \RR$, it is clear that $f(\bN) g(\bN) = g(\bN) f(\bN) =
(f g) (\bN)$ and $(f+g)(\bN) = f(\bN) + g(\bN)$.
Furthermore: $\ba \ba^\dag  =  \bN+\bf I$, $\ba^\dag \ba =  \bN$, $\ba f(\bN)  =  f(\bN+1) \ba$, $\ba^\dag f(\bN)  =  f(\bN-1) \ba^\dag$.

\subsection{Proof of Proposition~\ref{properties-pk}}\label{proof-properties-pk}

Fix any $\rho \in D_*$ and let   $n \in \mathbb{N}$. In particular, $\rho \ket{n} = \rho_{nn} \ket{n}$.
It then follows from \eqref{eq:M0}--\eqref{eq:Mm} that:
\[
\begin{array}{l}
	\bM_g(0) \rho \bM_g^\dag(0)\ket{n} = \rho_{nn} \cos^2 \left( \tfrac{\phi_0 n + \phi_R}{2} \right) \ket{n}, \smallskip \\
	\bM_e(0) \rho \bM_e^\dag(0)\ket{n} = \rho_{nn} \sin^2 \left( \tfrac{\phi_0 n + \phi_R}{2} \right)  \ket{n}, \smallskip \\
	\bM_g(+1) \rho \bM_g^\dag(+1)\ket{n} \hspace{-2pt} = \hspace{-2pt}
		\left\{
		\begin{array}{l}
			\hspace{-6pt} 0, \mbox{ for } n = 0, \smallskip \\
			\hspace{-6pt} \rho_{n-1,n-1} \hspace{-2pt}  \sin^2 \left( \tfrac{\theta_0}{2}\sqrt{n} \right) \ket{n}, n \geq 1,
		\end{array}
		\right. \smallskip \\			
	\bM_e(+1) \rho \bM_e^\dag(+1) \ket{n} = \rho_{nn} \cos^2 \left( \tfrac{\theta_0}{2}\sqrt{n+1} \right)  \ket{n}, \smallskip \\
	\bM_g(-1) \rho \bM_g^\dag(-1)\ket{n} = \rho_{nn} \cos^2 \left( \tfrac{\theta_0}{2}\sqrt{n} \right)  \ket{n}, \smallskip \\
	\bM_e(-1) \rho \bM_e^\dag(-1)\ket{n} = \rho_{n+1,n+1}\sin^2 \left( \tfrac{\theta_0}{2}\sqrt{n+1} \right)  \ket{n}.
\end{array}
\]
Therefore:
{\small
\begin{align}
	\label{zg} & \bM_g(0) \rho \bM_g^\dag(0) =   \sum_{n=n_{min}(\rho)}^{n_{max}(\rho)} \rho_{nn} \cos^2 \left( \tfrac{\phi_0 n + \phi_R}{2} \right) \ket{n} \bra{n}, \\
	\label{ze} & \bM_e(0) \rho \bM_e^\dag(0) =  \sum_{n=n_{min}(\rho)}^{n_{max}(\rho)} \rho_{nn} \sin^2 \left( \tfrac{\phi_0 n + \phi_R}{2} \right) \ket{n} \bra{n}, \\
	\label{pg} & \bM_g(+1) \rho \bM_g^\dag(+1) =  \hspace{-16pt} \sum_{n=n_{min}(\rho)+1}^{n_{max}(\rho)+1} \hspace{-16pt} \rho_{n-1,n-1}\sin^2 \left( \tfrac{\theta_0}{2}\sqrt{n} \right) \ket{n} \bra{n}, \\
	\label{pe} & \bM_e(+1) \rho \bM_e^\dag(+1) =  \hspace{-8pt}  \sum_{n=n_{min}(\rho)}^{n_{max}(\rho)} \hspace{-8pt}  \rho_{nn} \cos^2 \left( \tfrac{\theta_0}{2}\sqrt{n+1} \right) \ket{n} \bra{n}, \\
	\label{mg} & \bM_g(-1) \rho \bM_g^\dag(-1) =  \sum_{n=n_{min}(\rho)}^{n_{max}(\rho)} \rho_{nn} \cos^2 \left( \tfrac{\theta_0}{2}\sqrt{n} \right) \ket{n} \bra{n}, \\
	\label{me} & \bM_e(-1) \rho \bM_e^\dag(-1) = \hspace{-32pt}  \sum_{n=\max\{0,n_{min}(\rho)-1\}}^{n_{max}(\rho)-1} \hspace{-32pt}  \rho_{n+1,n+1}\sin^2 \left( \tfrac{\theta_0}{2}\sqrt{n+1} \right) \ket{n} \bra{n}.
\end{align}
}By assumption, $\rho_0 \in D_*$. Then, \eqref{eq:IdealMarkovModel}, \eqref{zg}--\eqref{me} above and induction on $k$ show the assertions in Proposition~\ref{properties-pk}.

\subsection{Computation of $Q_V(\rho, u)$}\label{aQV}
Fix any $\rho \in D_*$ and $\overline{n} \in \NN$. Recall that  $V (\rho ) = \tr {d (\bN) \rho}$, where $d$:~$\NN \rightarrow \RR$ be given by $d(n) = (n - \overline {n})^2$.
Note that \eqref{eq:IdealMarkovModel} implies that, for each $u \in \{-1,0,1\}$,
\begin{equation}\label{generalexpectation}
	 \begin{array}{l}
	 	\EE{V(\rho_{k+1})~|~\rho_k =\rho, u_k=u} \smallskip \\
		=  \tr{ \hspace{-2pt} d(\bN) \bM_g(u) \rho  \bM^\dag_g(u) \hspace{-2pt}} \hspace{-1pt} + \hspace{-1pt} \tr{\hspace{-2pt} d(\bN) \bM_e(u) \rho \bM^\dag_e(u)\hspace{-2pt}}\hspace{-2pt}.
	\end{array}
\end{equation}

Take $u=0$. From  \eqref{zg}--\eqref{ze} in Appendix~\ref{proof-properties-pk}, one has
\begin{align*}
	& \EE{V(\rho_{k+1})~|~\rho_k =\rho, u_k=0} \\
	& = \tr{ \hspace{-1pt}d(\bN) \bM_g(0) \rho  \bM^\dag_g(0)\hspace{-1pt}}  +  \tr{\hspace{-1pt} d(\bN) \bM_e(0) \rho \bM^\dag_e(0) \hspace{-1pt}} \\
	&  = \tr{ d(\bN) \left[ \bM_g(0) \rho  \bM^\dag_g(0) + \bM_e(0) \rho \bM^\dag_e(0)\right]} \\
	&  = \tr{ d(\bN) \rho} = V(\rho).
\end{align*}
In particular,
\begin{equation} \label{QV0}
	Q_V( \rho, 0) = 0.
\end{equation}

Now, take $u=+1$. Then, \eqref{generalexpectation} above and \eqref{pg}--\eqref{pe} in Appendix~\ref{proof-properties-pk} provide that
\begin{align*}
	& \EE{V(\rho_{k+1})~|~\rho_k=\rho, u_k=+1} \\
	& =  \tr{   {\sin^2\left(\tfrac{\theta_0}{2} \sqrt{\bN +1} \right) } d(\bN +1 )  \rho } \\
        & \, + \tr{{\cos^2\left(\tfrac{\theta_0}{2} \sqrt{\bN+1} \right)} d(\bN) \rho}.
\end{align*}
By summing and subtracting $\tr{{\sin^2\left(\tfrac{\theta_0}{2} \sqrt{\bN +1 } \right)} d(\bN) \rho}$,
\begin{align*}
	& \EE{V(\rho_{k+1})~|~\rho_k=\rho, u_k=+1} \\
	&  =  \tr{   d(\bN) \rho}  \\
	& \, + \tr{{\sin^2\left(\tfrac{\theta_0}{2} \sqrt{\bN +1} \right)  \left[ d(\bN +1) - d(\bN)  \right] \rho}} \\
        & = V(\rho) + \tr{{\sin^2\left(\tfrac{\theta_0}{2} \sqrt{\bN +1} \right)  \left[ d(\bN +1) - d(\bN) \right] \rho}}.
\end{align*}
In particular,
\begin{align}
\nonumber	& Q_V(\rho, +1) \hspace{-1pt} = \hspace{-1pt} - \hspace{-1pt} \tr{\hspace{-1pt} {\sin^2 \hspace{-2pt} \left(\tfrac{\theta_0}{2} \sqrt{\bN +1} \right) \hspace{-2pt} \left[ d(\bN +1) - d(\bN) \right] \hspace{-1pt} \rho}\hspace{-2pt}} \hspace{-1pt}\hspace{-1pt}, \\
\label{QVplus}  & = - \sum_{n \in \NN} \rho_{nn} \left[2(n - \overline{n}) + 1 \right]  \sin^2\left(\tfrac{\theta_0}{2} \sqrt{n +1} \right).
\end{align}


Finally, take $u=-1$. Using \eqref{generalexpectation} above and \eqref{mg}--\eqref{me} in Appendix~\ref{proof-properties-pk}, $\EE{V(\rho_{k+1})~|~\rho_k=\rho, u_k=-1} =  \tr{\hspace{-2pt} \sin^2 \hspace{-2pt} \left(\tfrac{\theta_0}{2} \sqrt{\bN} \right) \hspace{-2pt}  d(\bN-1)  \rho \hspace{-2pt}} +  \tr{\hspace{-2pt} {\cos^2 \hspace{-2pt} \left(\tfrac{\theta_0}{2} \sqrt{\bN} \right)} \hspace{-1pt} d(\bN) \rho \hspace{-2pt}}$.
By summing and subtracting $\tr{\sin^2\left(\tfrac{\theta_0}{2} \sqrt{ \bN } \right) d(\bN) \rho}$,
\begin{align*}
	& \EE{V(\rho_{k+1})~|~\rho_k=\rho, u_k=-1} \\
	& =  \tr{   d(\bN) \rho}  + \tr{{\sin^2\left(\tfrac{\theta_0}{2} \sqrt{\bN} \right)  \left[ d(\bN -1) - d(\bN) \right]} \rho} \\
	& = V(\rho) + \tr{{\sin^2\left(\tfrac{\theta_0}{2} \sqrt{\bN} \right)  \left[ d(\bN -1) - d(\bN) \right] \rho}}.
\end{align*}
In particular,
\begin{align}
\nonumber	Q_V(\rho, -1) &= -\tr{{\sin^2\left(\tfrac{\theta_0}{2} \sqrt{\bN} \right)  \left[ d(\bN - 1) - d(\bN) \right] \rho}}, \\
\label{QVneg}                         &= \sum_{n \in \NN} \rho_{nn} \left[2(n - \overline{n}) - 1 \right]  \sin^2\left(\tfrac{\theta_0}{2} \sqrt{n} \right).
\end{align}

\subsection{Proof of Lemma \ref{Lem1}}
 \label{aLem1}

Assume that $N_0$ is even (otherwise one may take $N_0+1$ instead of $N_0$ in this proof).
Define the function $\eta$:~$\NN \rightarrow \RR$ by
\begin{equation}\label{etadef}
   \eta(\ell) = \left[ \frac{2}{\theta_0} \left( \ell \frac{\pi}{2} + \frac{\pi}{4} \right) \right]^2.
\end{equation}
By definition, one has $\frac{\theta_0}{2}\sqrt{\eta(\ell)} = \ell \frac{\pi}{2} + \frac{\pi}{4}$
for all $\ell \in \NN$.
Let $h = \pi/4 - \arcsin \left(\sqrt{1/2-a} \right)$. Using the definition of $h$ and the symmetries\footnote{More precisely, $\sin^2(\pi/2 - x) = 1 - \cos^2(\pi/2 - x) = 1 - \sin^2(x)$.} of the function $\sin^2(\cdot)$, it is easy to
show that
\begin{equation}\label{eh}
	1/2 - a \leq \sin^2(x+\pi/4) \leq a+1/2, \quad \forall x \in [-h, h].
\end{equation}
Let $\overline{\ell} \in \NN$ be even and big enough such
that the following two conditions are simultaneously met:
\begin{equation}\label{ineqN}
	\eta(\overline{\ell})  >  N_0/2 + N,  \qquad
	\frac{1}{8} {\theta_0 N_0} / {\sqrt{\eta(\overline{\ell}) - N_0/2}} \leq  h.
\end{equation}
Now, take $\overline{N} = \lfloor \eta(\overline{\ell}) \rceil - \frac{N_0}{2} + 1 > N$, where
$\lfloor \eta \rceil$ denotes the greatest integer which is less or equal to $\eta$.
By construction, $\eta(\overline{\ell})$ is in-between the points
$\overline{N}+N_0/2 -1$ and $\overline{N} + N_0/2$, and hence it is in the interval $[\overline{N}, \overline{N} + N_0-1]$. Then, for $n=\overline{N}, \dots, \overline{N} + N_0-1$, one has that
$| n - \eta(\overline{\ell}) | < N_0/2 $. Consider the function $\phi(x) = \frac{\theta_0}{2} \sqrt{x}$. From the fact that $\phi^\prime (x)=  \frac{\theta_0}{4\sqrt{x}}$,
by the mean value theorem applied to the function $\phi$ and the second inequality in \eqref{ineqN}, one obtains
  \[
     \left| \frac{\theta_0}{2} \sqrt{n}  -  \frac{\theta_0}{2} \sqrt{\eta ( \overline{\ell} )} \right| < h , \quad \mbox{for } n = \overline{N}, \dots, \overline{N} + N_0-1.
  \]
Then, the proof follows easily from (\ref{etadef}), (\ref{eh}) and the fact that $\sin^2(x-\ell\pi/2) = \sin^2(x)$, for every even $\ell \in \NN$.

\subsection{Proof of Lemma \ref{Lem2}}\label{aLem2}

\noindent \textbf{Proof of the first claim:}  Let $u \in \{-1, 0, 1\}$, $\rho \in D_*$. Recall that $W(\rho) = - \sum_{n \in \NN} \rho_{nn}^2$. Since
$\tr{\rho} =  \sum_{n \in \NN} \rho_{nn} = 1$, then $-1 = - \sum_{n \in \NN} \rho_{nn} \leq W(\rho)  \leq 0$.
Now, by \eqref{eq:IdealMarkovModel},  $\EE{W(\rho_{k+1}) ~|~ \rho_k=\rho, u_k = u} = p_{g,k} W(\rho^{g}_{k+1}) +  p_{e,k} W(\rho^{e}_{k+1})$,
where $p_{g,k}, p_{e,k} \geq 0$ with $p_{g,k} + p_{e,k} = 1$. Thus $-1 \leq \EE{W(\rho_{k+1}) ~|~ \rho_k=\rho, u_k = u} \leq 0$.
Since $Q_W(\rho, u)$ is the difference of two numbers that are in-between $-1$ and $0$, one concludes that $|Q_W(\rho, u)| \leq 1$.

The second, third and fourth claims, are immediate from \eqref{QV0}, \eqref{QVplus} and \eqref{QVneg} in Appendix~\ref{aQV}, respectively.

\subsection{Proof of Lemma \ref{Lem3}}\label{aLem3}

\noindent \textbf{Proof of the first claim:}
Let $\rho \in D_{m_0}$. By \eqref{zg}--\eqref{ze} in Appendix~\ref{proof-properties-pk}, $\bM_g(0) \rho \bM_g^\dag(0) + \bM_e(0) \rho \bM_e^\dag(0) = \rho$.
Taking $\rho_k=\rho$ in $u_{k}=0$ in \eqref{eq:IdealMarkovModel}, define
\begin{align*}
	& \rho^y \triangleq \rho^y_{k+1} = \dfrac{\bM_y(0) \rho \bM_y^\dag(0)}{\tr{\bM_y(0) \rho \bM_y^\dag(0) }}, \quad \mbox{for } y=g,e.
\end{align*}
Hence, $\alpha \rho^g + (1-\alpha) \rho^e= \rho$, where
$\alpha \triangleq p_{g,k} = \tr{\bM_g(0) \rho \bM_g^\dag(0)}$. In particular,
$\alpha \rho^g_{nn} + (1-\alpha) \rho^e_{nn} = \rho_{nn}$, for $n \in \NN$. Note that, if $\alpha = 0$, then $\bM_g(0) \rho \bM_g^\dag(0) = 0$, and so $\rho^e=\rho$. Similarly, $\alpha = 1$ implies $\rho^g=\rho$. Thus, the identity
$\alpha \rho^g_{nn} + (1-\alpha) \rho^e_{nn} = \rho_{nn}$, for $n \in \NN$, still holds when $\alpha = 0$ or $\alpha = 1$.
From \eqref{eq:IdealMarkovModel}, \eqref{defVW} and $\alpha=p_{g,k}$, one has
\begin{align}\label{convexity}
\nonumber	 & Q_{W}(\rho,0)  = W(\rho) - \big[ p_{g,k} W(\rho^{g}_{k+1}) +  p_{e,k} W(\rho^{e}_{k+1}) \big]  \smallskip \\
\nonumber         & =  \sum_{n \in \NN} \alpha \left( \rho^g_{nn} \right)^2 + (1-\alpha)  \left( \rho^e_{nn} \right)^2 - \big[ \alpha  \rho^g_{nn} + (1-\alpha) \rho^e_{nn} \big]^2  \smallskip  \\
			 & = \alpha (1-\alpha) \sum_{n \in \NN}  \big[ \rho^g_{nn} -  \rho^e_{nn}\big]^2  \geq 0,
\end{align}
thereby showing the first part of the first claim.

If $\rho = \ket{m}\bra{m}$ for some $m \in \NN$ with $0 < \alpha < 1$, then  \eqref{zg}--\eqref{ze} in Appendix~\ref{proof-properties-pk} imply that $\rho^g = \rho^e = \rho$, and so $Q_{W}(\rho,0)= 0$. Now, one shows that $\rho = \ket{m}\bra{m}$ for some $m \in \NN$ whenever $Q_{W}(\rho,0) = 0$.  Suppose $Q_{W}(\rho,0) = 0$. Then, \eqref{convexity} implies that $\alpha = 0$, or $\alpha = 1$, or $\rho^g_{nn} = \rho^e_{nn}$ for all $n \in \NN$  with $0 < \alpha < 1$. Assume that $\alpha = 0$. Hence, $\bM_g(0) \rho \bM_g^\dag(0) = \sum_{n \in \NN} \rho_{nn} \cos^2 ( \tfrac{\phi_0 n + \phi_R}{2}) \ket{n} \bra{n} = 0$ by \eqref{zg} in Appendix~\ref{proof-properties-pk}. Suppose that
$\rho \neq \ket{m} \bra{m}$ for every $m \in \NN$. Thus, there exists $n_1, n_2 \in \NN$ with $n_1 \neq n_2$, $\rho_{n_1,n_1} > 0$, $\rho_{n_2,n_2} > 0$.
Recall that $\sin(x_1) = \pm \sin(x_2)$ if and only if $x_1 + x_2 = \ell \pi$ or $x_2 - x_1 = \ell \pi$, where $\ell$ is an integer.
Therefore, $\sin ( \tfrac{\phi_0 n_1 + \phi_R}{2} ) = \pm \sin ( \tfrac{\phi_0 n_2 + \phi_R}{2} )$, which
contradicts the assumptions that $\phi_0/ \pi$ is an irrational number and $\phi_R=\pi/2 - \bar n \phi_0$. One has shown that $\rho = \ket{m}\bra{m}$ for some $m \in \NN$ whenever $\alpha = 0$. If $\alpha = 1$, or $\rho^g_{nn} = \rho^e_{nn}$ for all $n \in \NN$ with $0 < \alpha <1$, then from similar arguments and computations one also concludes that $\rho = \ket{m}\bra{m}$ for some $m \in \NN$.

\noindent \textbf{Proof of the second claim:}
Let $m \in \NN$ and take $\rho = \ket{m}\bra{m} \in D_*$. It is clear that $W(\rho)=-\sum_{n \in \NN } \rho_{nn}^2 = -1$. From \eqref{pg}--\eqref{me} in Appendix~\ref{proof-properties-pk}, one has that:
\begin{equation}\label{Mequations}
\begin{array}{l}
	\big( \bM_g(+1) \rho \bM_g^\dag(+1) \big)_{nn} = \delta(n,m+1) \sin^2 \left( \tfrac{\theta_0}{2}\sqrt{m+1} \right), \medskip \\
	\big( \bM_e(+1) \rho \bM_e^\dag(+1) \big)_{nn} = \delta(n,m) \cos^2 \left( \tfrac{\theta_0}{2}\sqrt{m+1} \right), \medskip \\
	\big( \bM_g(-1) \rho \bM_g^\dag(-1) \big)_{nn} =  \delta(n,m) \cos^2 \left( \tfrac{\theta_0}{2}\sqrt{m} \right), \medskip  \\
	\big( \bM_e(-1) \rho \bM_e^\dag(-1) \big)_{nn} = \delta(n+1,m) \sin^2 \left( \tfrac{\theta_0}{2}\sqrt{m} \right),
\end{array}
\end{equation}
where $\delta(n,m)$ is the usual Kronecker delta: $\delta(n,m) = 0$ if $n \neq m$, and $\delta(n,m)=1$ if $n=m$. In particular:
\begin{align*}
	& \tr{\bM_g(+1) \rho \bM_g^\dag(+1)} = \sin^2 \left( \tfrac{\theta_0}{2}\sqrt{m+1} \right), \\
	& \tr{\bM_e(+1) \rho \bM_e^\dag(+1)} = \cos^2 \left( \tfrac{\theta_0}{2}\sqrt{m+1} \right),\\
	& \tr{\bM_g(-1) \rho \bM_g^\dag(-1)} = \cos^2 \left( \tfrac{\theta_0}{2}\sqrt{m} \right), \\
	& \tr{\bM_e(-1) \rho \bM_e^\dag(-1)} = \sin^2 \left( \tfrac{\theta_0}{2}\sqrt{m} \right), \\
	& \displaystyle \sum\limits_{n \in \NN} \left(\hspace{-1pt} \dfrac{\bM_y(u) \rho \bM_y^\dag(u)}{\tr{\bM_y(u) \rho \bM_y^\dag(u)}} \hspace{-1pt} \right)^2_{nn} \hspace{-8pt} = 1, \mbox{ for } u=\pm 1, \;
	y=g,e
\end{align*}
(assuming no division by 0).
Now, using \eqref{eq:IdealMarkovModel} and the above computations, one gets
\begin{align*}
	& \EE{W(\rho_{k+1}) ~|~ \rho_k=\rho, u_k = \pm 1}  \\
	& = p_{g,k} W(\rho^{g}_{k+1}) +  p_{e,k} W(\rho^{e}_{k+1}) \smallskip \\
	& = - \displaystyle \sum_{y=g,e} \left[ \tr{\bM_y(\pm 1) \rho \bM_y^\dag(\pm 1)} \times \right. \\
	& \hspace{48pt} \times \left. \sum_{n \in \NN} \left( \dfrac{\bM_y(\pm 1) \rho \bM_y^\dag(\pm 1)}{\tr{\bM_y(\pm 1) \rho \bM_y^\dag(\pm 1)}} \right)^2_{nn} \right] \\
	& = -1 = W(\rho).
\end{align*}
Therefore, $Q_W(\ket{m}\bra{m},\pm1) = 0$.

\subsection{Proof of Proposition~\ref{propertiesDelta}}\label{proofDelta}

Fix $\rho \in \mathbb{D}_*$. Since $\tr{ d(\bN) \rho} = \tr{ d(\bN) \Delta \rho}$ and $\rho_{nn}=(\Delta \rho)_{nn}$ for $n \in \NN$,
the first two assertions are immediate from the definitions. As for the third and fourth assertions, let $\ket{\psi} = \sum_{m=0}^\infty \langle m | \psi  \rangle \ket{m} \in \mathcal{H}$. Note that
$\rho \ket{m} = \sum_{n=0}^{n_{max}(\rho)} \rho_{mn} \ket{n}$, for $m \in \NN$. Using \eqref{eq:M0}--\eqref{eq:Mm}:
\begin{align*}
	 & \bM_g(0) \rho \bM_g^\dag(0) \ket{\psi} \\
	 & = \displaystyle \sum_{m,n=0}^{n_{max}(\rho)}  \rho_{mn} \cos \left( \tfrac{\phi_0 m + \phi_R}{2} \right) \cos \left( \tfrac{\phi_0 n + \phi_R}{2} \right) \langle m | \psi  \rangle \ket{n}, \\
	 & \bM_e(0) \rho \bM_e^\dag(0) \ket{\psi} \\
	 & =  \displaystyle \sum_{m,n=0}^{n_{max}(\rho)}  \rho_{mn}   \sin \left( \tfrac{\phi_0 m + \phi_R}{2} \right) \sin \left( \tfrac{\phi_0 n + \phi_R}{2} \right) \langle m | \psi  \rangle \ket{n}, \\
	 & \bM_g(+1) \rho \bM_g^\dag(+1)\ket{\psi} \\
	 & = \hspace{-10pt} \displaystyle \sum\limits_{m=1,n=0}^{n_{max}(\rho)+1} \hspace{-10pt} \rho_{m-1,n} \sin \left( \tfrac{\theta_0}{2}\sqrt{m} \right) \sin \left( \tfrac{\theta_0}{2}\sqrt{n+1} \right) \langle m | \psi  \rangle \ket{n+1}, \\
	& \bM_e(+1) \rho \bM_e^\dag(+1) \ket{\psi} \\
	& =  \displaystyle \sum_{m,n=0}^{n_{max}(\rho)} \rho_{mn} \cos \left( \tfrac{\theta_0}{2}\sqrt{m+1} \right) \cos \left( \tfrac{\theta_0}{2}\sqrt{n+1} \right) \langle m | \psi  \rangle \ket{n}, \\
	& \bM_g(-1) \rho \bM_g^\dag(-1) \ket{\psi} \\
	& = \displaystyle \sum_{m,n=0}^{n_{max}(\rho)} \rho_{mn}  \cos  \left( \tfrac{\theta_0}{2}\sqrt{m} \right) \cos  \left( \tfrac{\theta_0}{2}\sqrt{n} \right) \langle m | \psi  \rangle \ket{n}, \\
	& \bM_e(-1) \rho \bM_e^\dag(-1) \ket{\psi} \\
	& =   \hspace{-8pt} \displaystyle \sum\limits_{m=0,n=1}^{n_{max}(\rho)}  \hspace{-8pt} \rho_{m+1,n}   \sin \left( \tfrac{\theta_0}{2}\sqrt{m+1} \right) \sin \left( \tfrac{\theta_0}{2}\sqrt{n} \right) \langle m | \psi  \rangle \ket{n-1}.
\end{align*}
Since $\Delta \rho \in D_* \subset \mathbb{D}_*$, $n_{max}(\Delta\rho) = n_{max}(\rho)$ and $(\Delta \rho)_{nn} = \rho_{nn}$, the proof is straightforward from \eqref{zg}--\eqref{me} in Appendix~\ref{proof-properties-pk}.

\bibliographystyle{plain}        

\begin{thebibliography}{1}

\bibitem{AminiSDSMR2013A}
H.~Amini, R.A. Somaraju, I.~Dotsenko, C.~Sayrin, M.~Mirrahimi, and P.~Rouchon.
\newblock Feedback stabilization of discrete-time quantum systems subject to
  non-demolition measurements with imperfections and delays.
\newblock {\em Automatica}, 49(9):2683--2692, September 2013.

\bibitem{brune-et-al:PhRevA92}
M.~Brune, S.~Haroche, J.-M. Raimond, L.~Davidovich, and N.~Zagury.
\newblock Manipulation of photons in a cavity by dispersive atom-field
  coupling: Quantum-nondemolition measurements and generation of
 ``{S}chr\"{o}dinger cat'' states.
\newblock {\em Physical Review A}, 45(7):5193--5214, 1992.

\bibitem{conway-book}
J.~B. Conway.
\newblock {\em A Course in Operator Theory}.
\newblock American Mathematical Society, 2000.

\bibitem{haroche-raimondBook06}
S.~Haroche and J.M. Raimond.
\newblock {\em Exploring the Quantum: Atoms, Cavities and Photons.}
\newblock Oxford University Press, 2006.

\bibitem{kushner-71}
H.J. Kushner.
\newblock {\em Introduction to Stochastic Control}.
\newblock Holt, Rinehart and Wilson, INC., 1971.

\bibitem{nielsen-chang-book}
M.A. Nielsen and I.L. Chuang.
\newblock {\em Quantum Computation and Quantum Information}.
\newblock Cambridge University Press, 2000.

\bibitem{reed-simon-1}
M.~Reed and B.~Simon.
\newblock {\em Methods of Modern Mathematical Physics: Functional Analysis
  (Vol. 1)}.
\newblock Academic Press, 1980.

\bibitem{zhouPRL2012}
X.~Zhou, I.~Dotsenko, B.~Peaudecerf, T.~Rybarczyk, C.~Sayrin, J.M.~Raimond
  S.~Gleyzes, M.~Brune, and S.~Haroche.
\newblock Field locked to {F}ock state by quantum feedback with single photon
  corrections.
\newblock {\em Physical Review Letter}, 108:243602, 2012.

\end{thebibliography}

\end{document}